\documentclass[preliminary,copyright,creativecommons]{eptcs}
\providecommand{\event}{ACT 2021} 
\usepackage{breakurl}             
\usepackage{underscore}           
\usepackage{amssymb,amsmath}
\usepackage{amsthm,thmtools}
\usepackage{hyperref}
\usepackage{cleveref}
\usepackage{multirow}
\usepackage{bm}
\usepackage{xcolor}
\hypersetup{colorlinks=true}

\usepackage{tikz}
\usetikzlibrary{arrows,arrows.meta,decorations.pathreplacing,decorations.markings,shapes,calc}
\tikzset{labelsize/.style={font=\scriptsize}}
\tikzset{string/.style={very thick}}
\tikzset{
  pto/.style={->,postaction={decorate},
    decoration={
        markings,
        mark=at position 0.5 with {\arrow{|}}}
  },
}
\tikzset{2cell/.style={-implies,double,double equal sign distance,shorten >=9pt, shorten <=10pt}}

\mathchardef\mhyphen="2D

\declaretheorem[style=plain,numberwithin=section,name=Theorem]{theorem}
\declaretheorem[style=plain,sibling=theorem,name=Lemma]{lemma}
\declaretheorem[style=plain,sibling=theorem,name=Proposition]{proposition}
\declaretheorem[style=plain,sibling=theorem,name=Corollary]{corollary}

\declaretheorem[style=definition,qed=$\blacksquare$,sibling=theorem,name=Definition]{definition}
\declaretheorem[style=definition,qed=$\blacksquare$,sibling=theorem,name=Example]{example}
\declaretheorem[style=definition,qed=$\blacksquare$,sibling=theorem,name=Remark]{remark}

\makeatletter
\newcommand{\pto}{}
\newcommand{\pgets}{}

\DeclareRobustCommand{\pto}{\mathrel{\mathpalette\p@to@gets\to}}
\DeclareRobustCommand{\pgets}{\mathrel{\mathpalette\p@to@gets\gets}}

\newcommand{\p@to@gets}[2]{%
  \ooalign{\hidewidth$\m@th#1\mapstochar\mkern5mu$\hidewidth\cr$\m@th#1\to$\cr}%
}
\makeatother

\newcommand{\Pow}{\mathcal{P}}
\newcommand{\FinSet}{\mathbf{FinSet}}
\newcommand{\Set}{\mathbf{Set}}
\newcommand{\id}{\mathrm{id}}
\newcommand{\defemph}[1]{\textbf{#1}}
\newcommand{\Pol}{\mathrm{Pol}}

\newcommand{\Rbar}{\overline{\mathbb{R}}}

\newcommand{\ob}{\mathrm{ob}}

\newcommand{\cat}[1]{\mathcal{#1}}

\newcommand{\Sol}{\mathcal{S}}
\newcommand{\Opt}{\mathcal{O}}

\newcommand{\TVCSP}{\mathrm{TVCSP}}
\newcommand{\CSP}{\mathrm{CSP}}
\DeclareMathOperator{\dom}{dom}

\crefname{theorem}{Theorem}{Theorems}
\crefname{proposition}{Proposition}{Propositions}
\crefname{lemma}{Lemma}{Lemmas}
\crefname{exmp}{Example}{Examples}
\crefname{corollary}{Corollary}{Corollarys}
\crefname{claim}{Claim}{Claims}
\crefname{remark}{Remark}{Remarks}
\crefname{section}{Section}{Sections}
\crefname{definition}{Definition}{Definitions}
\crefname{example}{Example}{Examples}
\crefname{table}{Table}{Tables}
\crefname{appendix}{Appendix}{Appendices}

\title{
Quantaloidal Approach to Constraint Satisfaction\footnote{Most of the work was done while Kei Kimura was at Saitama University.
}}
\author{Soichiro Fujii\thanks{Supported by JST, ERATO Grant Number JPMJER1603 (HASUO Metamathematics for Systems Design Project).}
\institute{Research Institute for Mathematical Sciences\\
Kyoto University\\
Kyoto, Japan}
\email{s.fujii.math@gmail.com}
\and
Yuni Iwamasa\thanks{Supported by JSPS KAKENHI Grant Numbers 20K23323, 20H05795.}
\institute{Department of Communications and Computer Engineering\\
Graduate School of Informatics\\
Kyoto University\\
Kyoto, Japan}
\email{iwamasa@i.kyoto-u.ac.jp}
\and
Kei Kimura\thanks{Supported by JST, ACT-X Grant Number JPMJAX200C, Japan, and JSPS KAKENHI Grant Numbers JP19K22841, JP21K17700.}
\institute{Faculty of Information Science and Electrical Engineering\\
Kyushu University\\
Fukuoka, Japan}
\email{kkimura@inf.kyushu-u.ac.jp}
}
\def\titlerunning{Quantaloidal Approach to Constraint Satisfaction}
\def\authorrunning{S. Fujii, Y. Iwamasa \& K. Kimura}
\begin{document}
\maketitle

\begin{abstract}
The constraint satisfaction problem (CSP) is a computational problem that includes a range of important problems in computer science.
We point out that fundamental concepts of the CSP, such as the solution set of an instance and polymorphisms, can be formulated abstractly inside the 2-category $\mathcal{P}\mathbf{FinSet}$ of finite sets and sets of functions between them. 
The 2-category $\mathcal{P}\mathbf{FinSet}$ is a quantaloid, and the formulation relies mainly on structure available in any quantaloid. This observation suggests a formal development of generalisations of the CSP and concomitant notions of polymorphism in a large class of quantaloids. 
We extract a class of optimisation problems as a special case, and show that their computational complexity can be classified by the associated notion of polymorphism.
\end{abstract}

\section{Introduction}
\label{sec:introduction}
\subsection{Background}
\label{subsec:background}
The \emph{constraint satisfaction problem} (CSP) is a computational problem of determining whether it is possible to assign values to variables while satisfying all given constraints.
The CSP provides a general framework capturing a variety of problems in diverse fields such as artificial intelligence (see, e.g., \cite{Tsa93,Dec03,RvW06}), theoretical computer science (e.g., \cite{CKS01}) and operations research (e.g., \cite{HOv18}), and has been studied from both practical and theoretical points of view.
Many heuristic algorithms have been developed and incorporated into CSP solvers, and these solvers are used 
for various purposes including corporate decision making (e.g., \cite{RvW06,BHZ06,LRSV18}).
Generalisations of the CSP are also widely studied; these include optimisation problems (e.g., \cite{Ziv12}) and counting problems (e.g., \cite{CaC17}).

Formally, a \defemph{CSP instance} $I=(V,D,\cat{C})$ is given by a finite set $V$ of \defemph{variables}, a finite set $D$ called the \defemph{domain}, and a finite set $\cat{C}$ of constraints. Here, each \defemph{constraint} is a triple $(k,\mathbf{x},\rho)$ consisting of a natural number $k$ called the \defemph{arity}, a $k$-tuple $\mathbf{x}\in V^k$ of variables called the \defemph{constraint scope}, and a $k$-ary relation $\rho\subseteq D^k$ on $D$ called the \defemph{constraint relation} of the constraint.
A \defemph{solution} of $I$ is a function $s\colon V\to D$ satisfying all constraints, i.e., such that for each $(k,\mathbf{x},\rho)\in\cat{C}$ with $\mathbf{x}=(x_1,\dots,x_k)$, we have $s(\mathbf{x})=(s(x_1),\dots,s(x_k))\in\rho$.
The set of all solutions of $I$ is denoted by $\Sol(I)$, and is called the \defemph{solution set} of $I$.
To solve the CSP instance $I$ is to output ``yes'' if there exists a solution of $I$, 
and ``no'' otherwise.

The following are two typical problems that can be modelled as CSPs.
\begin{example}
The Boolean satisfiability  problem (SAT) is the problem of determining whether a given conjunctive normal form (CNF) propositional formula is satisfiable or not.
Here, a \emph{CNF formula} is a conjunction of clauses, a \emph{clause} being a disjunction of literals, and a \emph{literal} being a propositional variable or its negation. For example, 
$(x_1 \vee \overline{x}_2 \vee x_4) \wedge (\overline{x}_1 \vee \overline{x}_2 \vee x_2 \vee x_{4}) \wedge (x_2 \vee x_4)$ 
is a CNF formula.
A CNF formula $\varphi(x_1,\dots,x_n)$ can be thought of as a CSP instance $I_\varphi$ with $V=\{x_1,\dots,x_n\}$ and $D=\{0,1\}$; each clause $\psi(x_{i_1},\dots,x_{i_k})$ of $\varphi$ gives rise to a constraint of $I_\varphi$ expressing the condition for a truth value assignment to $\{x_{i_1},\dots,x_{i_k}\}$ to make $\psi$ true. 
For example, the clause $(x_1\vee \overline{x}_2\vee x_4)$ corresponds to the constraint $(3, (x_1,x_2,x_4), \{0,1\}^3\setminus\{(0,1,0)\})$.

A well-known subclass of SAT is 3-SAT, in which the input is restricted to a \emph{3-CNF formula}, i.e., a CNF formula such that every clause is a disjunction of three literals. SAT and 3-SAT are fundamental in computer science; for example they are among the first problems shown to be NP-complete~\cite{Coo71,Lev73}. 
\end{example}
\begin{example}
\label{ex:graph-colouring}
Let $k$ be a positive integer.
In the graph $k$-colouring problem, we are given a simple undirected graph, i.e., a pair consisting of a finite set $V$ of vertices and a symmetric and irreflexive binary relation $E$ on $V$ representing the adjacency relation.
Our task is to determine whether it is possible to assign $k$ colours to the vertices so that adjacent vertices are assigned different colours.
The graph $k$-colouring problem is intensively studied in combinatorics (e.g., \cite{HeN04,BoM08}).
It is known that the graph $k$-colouring problem is in P (solvable in polynomial time) if $k \le 2$, and is NP-complete if $k \ge 3$.
To formulate the graph $k$-colouring problem as a CSP, let $D_k$ be the $k$-element set of colours and ${\neq_k} = \{ (d_1,d_2) \in D_{k}^2 \mid d_1 \neq d_2  \}$.
An instance $(V,E)$ of the graph $k$-colouring problem can be cast as the CSP instance $(V,D_k,\cat{C})$, where $\cat{C} = \{\, (2,(v_1,v_2),\neq_k) \mid (v_1,v_2) \in E \,\}$.
\end{example}

A notable theoretical result in this field is
the \emph{dichotomy theorem} for CSPs~\cite{Sch78,Bul17,Zhu20}.
To state the theorem, we need some definitions.
A \defemph{constraint language}
is a pair $(D,\cat{D})$
of a finite set $D$
and a finite family $\cat{D} = (\cat{D}_k)_{k \in\mathbb{N}}\in \prod_{k\in\mathbb{N}}\Pow(\Pow(D^k))$
of relations on $D$ (i.e., $\cat{D}_k$ is empty for all but finitely many $k$).
Each constraint language $(D,\cat{D})$ determines the class $\CSP(\cat{D})$ consisting of all CSP instances $(V,D',\cat{C})$
such that
$D'=D$ and,
for each constraint $(k,\mathbf{x},\rho)\in \cat{C}$,
we have 
$\rho\in\cat{D}_k$.
For example, $\CSP(\cat{D})$ reduces to 3-SAT when $(D,\cat{D})=\left(\{0,1\}, \big\{\,\{0,1\}^3\setminus \{(d_1,d_2,d_3)\}\mid d_1,d_2,d_3\in\{0,1\}\,\big\}\right)$, and to the graph $k$-colouring problem (for possibly directed graphs with loops) when $(D,\cat{D})=(D_k,\{\neq_k\})$. 
Roughly, the dichotomy theorem states that $\CSP(\cat{D})$ is in P if $\cat{D}$ satisfies a certain property, and is NP-complete otherwise.
Notice that this \emph{dichotomy} result is highly nontrivial, given the fact that under the assumption P $\neq$ NP, there exists an infinite hierarchy of complexity classes (up to a polynomial time reduction) containing P and contained in NP~\cite{Ladner}.
An interesting aspect of the dichotomy theorem is the fact that the border between P and NP-completeness can be captured by a purely algebraic criterion based on the notion of \emph{polymorphism}, to which we now turn.

In the long chain of research devoted to the analysis of computational complexity of $\CSP(\cat{D})$ (e.g., \cite{Sch78,JCG97,Jeavons-algebraic,Feder-Vardi,BJK05,Bul06,Bul17,Zhu20}), special attention has been paid to the \emph{symmetry} of problems. The idea is that a problem should be easy to solve if it admits enough symmetry. It is clear (at least intuitively) that if a CSP instance has certain symmetry, then so does its solution set. For example, the graph $k$-colouring problem is invariant under an arbitrary permutation of colours, and thus it follows that so is the solution set of each of its instances. 
Although the symmetry of a mathematical object is often captured by its group of automorphisms, this is not sufficient for the analysis of CSPs; for example, while the graph $k$-colouring problem admits the maximum automorphism group, it is NP-complete if $k\geq 3$.
It turns out that we need to enlarge the \emph{group of automorphisms} to the \emph{clone of polymorphisms}; here, a polymorphism of a mathematical object $X$ refers to a homomorphism from its finite power $X^n$ to $X$.\footnote{Our usage of the term ``polymorphism'' follows a tradition in universal algebra (see, e.g., \cite{Poschel-general}). In particular, it has nothing to do with \emph{polymorphism} in type theory and programming language theory, nor with \emph{poly-morphism} in the sense of \cite{Mochizuki}, which incidentally refers to a morphism in the free quantaloid $\Pow\cat{A}$ over a category $\cat{A}$ (see \cref{ex:PC}).}

The adequacy of polymorphisms in the current context is well-attested by their crucial use in a precise statement of the dichotomy theorem (\cref{thm:dichotomy-CSP}): given a constraint language $(D,\cat{D})$, $\CSP(\cat{D})$ is in P if the relational structure $(D,(\rho)_{k\in\mathbb{N},\rho\in\cat{D}_k})$ admits a \emph{Siggers operation} (see \cref{def:Siggers}) as a polymorphism, and is NP-complete otherwise.

\subsection{Our results}
\label{subsec:our-results}
In this paper, we shed a new light on the CSP and its variants by formulating their fundamental concepts in suitable \emph{quantaloids}. A quantaloid is a particularly well-behaved 2-category, in which right extensions and right liftings (right adjoints of precomposition and postcomposition by a morphism) always exist.

First we capture the ordinary CSP in the quantaloid $\Pow\FinSet$ whose objects are finite sets, whose morphisms $A\pto B$ are \emph{sets} of functions $A\to B$, and whose 2-cells are given by the inclusion relation. 
Observe that a $k$-ary relation $\rho$ on a finite set $D$ can be seen as a morphism $\rho\colon [k]=\{1,\dots,k\}\pto D$ in $\Pow\FinSet$. Thus each constraint $(k,\mathbf{x},\rho)\in\cat{C}$ of a CSP instance $I=(V,D,\cat{C})$ gives rise to the solid arrows in the diagram below, from which we obtain the right extension $\rho\swarrow \{\mathbf{x}\}$:
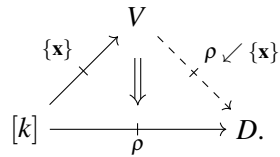
\[
\begin{tikzpicture}[baseline=-\the\dimexpr\fontdimen22\textfont2\relax ]
\node (L) at (0,0) {$[k]$};
\node (L1) at (1.5,1.5) {$V$};
\node (L2) at (3,0) {$D$.};

\draw [pto] (L) to node [auto,labelsize] {$\{\mathbf{x}\}$} (L1);
\draw [pto] (L) to node [auto,swap,labelsize] {$\rho$} (L2);
\draw [pto,dashed] (L1) to node [auto,labelsize] {$\rho\swarrow \{\mathbf{x}\}$} (L2);

\draw [2cell] (1.5,1.3) to (1.5,0);
\end{tikzpicture}
\]
As we shall see, $\rho\swarrow \{\mathbf{x}\}\colon V\pto D$ is precisely the \emph{set of all functions $V\to D$ satisfying the constraint $(k,\mathbf{x},\rho)$}. Thus the solution set $\Sol(I)$ can be seen as the morphism $V\pto D$ in $\Pow\FinSet$ expressed as
$
\Sol(I)=\bigcap_{(k,\mathbf{x},\rho)\in\cat{C}} \rho\swarrow\{\mathbf{x}\}.
$

We also give a quantaloidal formulation of polymorphisms. An $n$-ary polymorphism of a $k$-ary relation $\rho$ on $D$ is a homomorphism of relational structures $(D,(\rho))^n\to (D,(\rho))$.
The set $\Pol(\rho)_n$ of all $n$-ary polymorphisms of $\rho$ can be seen as a morphism $D^n\pto D$ in $\Pow\FinSet$, and is expressed as $\Pol(\rho)_n=\rho\swarrow (\{\pi_i\}_{i=1}^n\searrow \rho)$, where $\searrow$ denotes right lifting and the morphism $\{\pi_i\}_{i=1}^n\colon D^n\pto D$ is the set of projections. This provides a novel view to the set of polymorphisms of $\rho$ as the ``double dualisation'' with respect to $\rho$ of the set of projections.

This quantaloidal reformulation of the ordinary CSP opens the way to a formal definition of the \emph{quantaloidal CSP} and the associated notion of polymorphism in an abstract setting. We shall sketch such a definition in an arbitrary quantaloid of the form $\cat{QA}$ (see \cref{ex:QA}) generated by a quantale (one-object quantaloid) $\cat{Q}$ and a locally small category $\cat{A}$ with finite products. In particular, we show in this generality the claim that the solution set $\Sol(I)$ inherits the symmetry of an instance $I$, formulated suitably in terms of \emph{$\cat{Q}$-valued polymorphisms} (\cref{prop:Pol-closure-properties-QA}).

We then instantiate this general framework by setting $\cat{QA}=\Rbar\FinSet$ or $\Rbar \Set$, where $\Rbar$ is a quantale of extended real numbers. In these cases, the quantaloidal CSP contains a certain class of optimisation problems which we call the \emph{tropical valued CSP} (TVCSP). The TVCSP is different from the more widely studied optimisation variant of the CSP called the \emph{valued CSP} (VCSP), but is able to formulate certain scheduling problems as well as (via its infinitary variant) important concepts in continuous optimisation, such as quasiconvex functions and piecewise-linear convex functions.
We shall establish a dichotomy theorem for TVCSPs (satisfying suitable finiteness conditions) by reducing it to the dichotomy for CSPs. 
The border between P and NP-hardness for TVCSPs can be captured by the notion of $\Rbar$-valued polymorphism, which is a special case of our general notion of $\cat{Q}$-valued polymorphism.

\paragraph{Related work.}
Interaction between category theory and the field of algorithms and computational complexity is rare, and to the best of our knowledge this is the first paper relating the CSP and 2-category theory. We cite \cite{KOWZ20} as a recent paper applying categorical ideas to a generalisation of the CSP called the \emph{promise CSP}, although its approach and goal are entirely different from ours. 
A generalisation of the CSP valued in a certain class of idempotent semirings has been introduced in~\cite{BMR97,BMRSVF99}.
Whereas quantales are an infinitary version of idempotent semirings, our definition of $\Sol(I)$ in the quantaloidal CSP is incomparable with the corresponding notion (called \emph{consistency level}) in their framework.
{In a more recent work \cite{HMV17}, polymorphisms in the context of the above semiring-based generalised CSP are considered. 
Over a fixed constraint language, their notion of instance can capture a wider class of problems than ours, but as a consequence, their computational complexity results rely on extra assumptions; as the valuation structure (corresponding to our $\cat{Q}$), they adopt totally ordered commutative monoids whose unit element is the largest element and satisfying certain finiteness conditions.
We have not been able to find any clear relationship between their notion of polymorphism and ours.}

\paragraph{Outline.}
The remainder of this paper is organised as follows.
In \cref{sec:quantaloids} we recall the notion of quantaloid. In \cref{sec:polymorphism-quantaloid} we give a quantaloidal formulation of the CSP, and then proceed in \cref{sec:CSP-in-QA} to its generalisation in quantaloids of the form $\cat{QA}$. 
\cref{sec:polymorphism-tropical} is devoted to the TVCSP; we introduce it as a special case of the quantaloidal CSP, analyse its computational complexity and explore examples. 

\paragraph{Acknowledgements.} 
We thank Stanislav \v{Z}ivn\'{y} for providing a comment on an early draft of this paper and calling our attention to~\cite{HMV17},
and Takehide Soh for bibliographical information on CSP solvers~\cite{RvW06,BHZ06,LRSV18}.

\section{Quantaloids}
\label{sec:quantaloids}
In this section we review the definition and basic structure of quantaloids. See, e.g., \cite{Rosenthal-quantaloid,Stubbe-quantaloid-dist} for more information on quantaloids.
\begin{definition}
A \defemph{quantaloid} is a locally small category $\cat{K}$ equipped with a partial order $\leq_{A,B}$ on each hom-set $\cat{K}(A,B)$ such that 
\begin{itemize}
    \item for each $A,B\in\cat{K}$, $(\cat{K}(A,B),\leq_{A,B})$ is a complete lattice;
    \item for each $A,B,C\in\cat{K}$, the composition law $\cat{K}(B,C)\times\cat{K}(A,B)\to\cat{K}(A,C)$ preserves arbitrary suprema in each variable: for each set $J$ and $\varphi\colon A\to B$, $(\varphi_j\colon A\to B)_{j\in J}$, $\psi\colon B\to C$  and $(\psi_j\colon B\to C)_{j\in J}$ in $\cat{K}$, we have 
    \[
    \big(\bigvee_{j\in J}\psi_j\big)\circ \varphi=\bigvee_{j\in J}(\psi_j\circ \varphi)\quad\textrm{and}\quad    \psi\circ \big(\bigvee_{j\in J}\varphi_j\big)=\bigvee_{j\in J}(\psi\circ \varphi_j).
    \]
\end{itemize}

A quantaloid whose set of objects is a singleton is called a \defemph{quantale}. Explicitly, a quantale is a tuple $\cat{Q}=(Q,\leq,e,\otimes)$ such that $(Q,\leq)$ is a complete lattice, $(Q,e,\otimes)$ is a monoid, and the multiplication $\otimes\colon Q\times Q\to Q$ preserves arbitrary suprema in each variable.
\end{definition}

Note that a quantaloid $\cat{K}$ can be regarded as a 2-category, whose 2-cells are given by the partial order relation $\leq_{A,B}$ on each hom-set (that is, there exists a necessarily unique 2-cell $\varphi\Rightarrow \varphi'$ between a parallel pair of morphisms $\varphi,\varphi'\colon A\to B$ precisely when $\varphi\leq_{A,B}\varphi'$).
Each quantaloid is a \emph{biclosed} 2-category, meaning that right (Kan) extensions and right liftings always exist:

\begin{proposition}
Let $\cat{K}$ be a quantaloid. For each $A,B,C\in\cat{K}$, $\varphi\colon A\to B$ and $\psi\colon B\to C$ in $\cat{K}$, both 
\begin{equation}
\label{eqn:quantaloid-composition}
(-)\circ \varphi\colon \cat{K}(B,C)\to \cat{K}(A,C) \quad\textrm{and}\quad \psi\circ (-)\colon \cat{K}(A,B)\to\cat{K}(A,C)
\end{equation}
have right adjoints.
\end{proposition}

The right adjoints of \eqref{eqn:quantaloid-composition} are denoted by 
\[
(-)\swarrow \varphi\colon \cat{K}(A,C)\to \cat{K}(B,C) \quad\textrm{and}\quad \psi\searrow (-)\colon \cat{K}(A,C)\to\cat{K}(A,B)
\]
respectively. If $\theta\colon A\to C$ is a morphism in $\cat{K}$, then the morphism $\theta\swarrow \varphi\colon B\to C$ is called the \defemph{right extension} of $\theta$ along $\varphi$, and $\psi\searrow \theta\colon A\to B$ the \defemph{right lifting} of $\theta$ along $\psi$.
By the adjointness we have 
\begin{equation}
\label{eqn:adjointness-quantaloid}
\psi\leq_{B,C} \theta\swarrow \varphi \iff 
\psi\circ \varphi\leq_{A,C} \theta\iff 
\varphi\leq_{A,B} \psi\searrow \theta. 
\end{equation}
The following are formal properties of these operations in a quantaloid which we shall use later.
\begin{proposition}
\label{prop:quantaloid-basic-equations}
Let $\cat{K}$ be a quantaloid.
\begin{enumerate}
    \item
    For each set $J$ and $\varphi\colon A\to B$, $(\varphi_j\colon A\to B)_{j\in J}$, $\psi\colon B\to C$, $(\psi_j\colon B\to C)_{j\in J}$, $\theta\colon A\to C$ and $(\theta_j\colon A\to C)_{j\in J}$ in $\cat{K}$, we have  
    \begin{align*}
    \big(\bigwedge_{j\in J}\theta_j\big)\swarrow \varphi &= \bigwedge_{j\in J} (\theta_j\swarrow \varphi),& \theta\swarrow\big(\bigvee_{j\in J}\varphi_j\big)  &= \bigwedge_{j\in J} (\theta\swarrow \varphi_j),\\
    \psi\searrow \big(\bigwedge_{j\in J}\theta_j\big)&=\bigwedge_{j\in J}(\psi\searrow \theta_j),& \big(\bigvee_{j\in J}\psi_j\big)\searrow \theta&=\bigwedge_{j\in J}(\psi_j\searrow \theta).
    \end{align*}
    \item For each $\varphi\colon A\to B$, $\psi\colon B\to C$, $\theta\colon C\to D$ and $\gamma\colon A\to D$ in $\cat{K}$, we have 
    \[
    \gamma\swarrow (\psi\circ \varphi)=(\gamma\swarrow \varphi)\swarrow \psi,\qquad
    (\theta\circ \psi)\searrow \gamma=\psi\searrow (\theta\searrow \gamma),\qquad
    \theta\searrow (\gamma\swarrow \varphi)=(\theta\searrow \gamma)\swarrow \varphi.
    \]
\end{enumerate}
\end{proposition}
\begin{proof}
These are all straightforward consequences of the adjointness \eqref{eqn:adjointness-quantaloid}. As an example, we shall prove the last equation.  
For any $\psi'\colon B\to C$ in $\cat{K}$, we have 
\begin{align*}
\psi'\leq_{B,C}\theta\searrow (\gamma\swarrow \varphi) &\iff \theta\circ \psi'\leq_{B,D}\gamma\swarrow \varphi \\
&\iff (\theta\circ \psi')\circ \varphi\leq_{A,D}\gamma\\
&\iff \theta\circ (\psi'\circ \varphi)\leq_{A,D}\gamma \\
&\iff \psi'\circ \varphi\leq_{A,C}\theta\searrow \gamma \\
&\iff \psi'\leq_{B,C}(\theta\searrow \gamma)\swarrow \varphi. 
\end{align*}
Since $\psi'$ was  arbitrary, we must have $\theta\searrow (\gamma\swarrow \varphi)=(\theta\searrow \gamma)\swarrow \varphi$.
\end{proof}

We conclude this section with several examples of quantales and quantaloids.

\begin{example}
\label{ex:quantale-2}
The two-element quantale $\mathbf{2}=(\{0,1\},\leq, 1,\wedge)$ consists of the two-element chain (with $0\leq 1$) equipped with the monoid structure given by infima (or conjunction).
Since the multiplication $\wedge$ is commutative, right extensions and right liftings coincide, and are given by implication.
\end{example}

\begin{table}[t]
\centering
\begin{tabular}{c c|c c c}
  \multicolumn{2}{c|}{\multirow{2}{*}{$\beta+\alpha$}} &  & $\alpha$ & \\
 \multicolumn{2}{c|}{} & $\infty$  & $s$       & $-\infty$ \\ 
 \hline
    & $\infty$     & $\infty$  & $\infty$  & $\infty$ \\ 
 $\beta$&$t$           & $\infty$  &  $t+s$    & $-\infty$ \\
    &$-\infty$     & $\infty$ & $-\infty$ & $-\infty$ 
\end{tabular}
\qquad\qquad
\begin{tabular}{c c|c c c} 
 \multicolumn{2}{c|}{\multirow{2}{*}{$\gamma-\beta$}} &  & $\beta$ & \\
 \multicolumn{2}{c|}{} & $\infty$  & $t$       & $-\infty$ \\ 
 \hline
 &$\infty$      & $-\infty$  & $\infty$  & $\infty$ \\ 
$\gamma$ &$u$           & $-\infty$  &  $u-t$    & $\infty$ \\
 &$-\infty$     & $-\infty$ & $-\infty$ & $-\infty$ 
\end{tabular}
\caption{The operation tables for $\beta+\alpha$ and $\gamma-\beta\ (= \gamma\swarrow \beta = \beta\searrow \gamma)$ in $\Rbar$. The symbols $s,t$ and $u$ denote real numbers.}
\label{table:Rbar}
\end{table}

\begin{example}[\cite{Lawvere-state}]
\label{ex:quantale-R}
The quantale $\Rbar=(\mathbb{R}\cup\{\pm\infty\},\geq, 0, + )$ of extended real numbers consists of the totally ordered set of real numbers ordered by the \emph{opposite} $\geq$ of the usual order $\leq$, extended with the \emph{least} element $\infty$ and the \emph{greatest} element $-\infty$, and equipped with (an extension of) the usual addition. We have chosen the direction $\geq$ in order to match the standard setting of the VCSP, in which the goal is usually to \emph{minimise} a certain quantity rather than to \emph{maximise} it. To avoid confusion, we shall use the notations $\inf$ and $\sup$ to denote infima and suprema in $\mathbb{R}\cup\{\pm\infty\}$ with respect to the usual order $\leq$ (so that, e.g., $\inf\{3,5\}=3$). As a consequence, when we specialise certain formulas for general quantales to the quantale $\Rbar$, $\bigvee$ is translated to $\inf$ and $\bigwedge$ to $\sup$.
The requirement that $+$ should preserve arbitrary $\bigvee$ ($=$ $\inf$) in each variable determines its extension to $\Rbar\cup\{\pm\infty\}$ uniquely \cite{Lawvere-state}. The right extensions and right liftings coincide and are given by a suitable extension of subtraction; see \cref{table:Rbar}. 

Variants of $\Rbar$ may be obtained for example by restricting it to non-negative numbers \cite{Lawvere-metric} or to integers.
\end{example}

\begin{example}[\cite{Rosenthal-free}]
\label{ex:PC}
For any locally small category $\cat{A}$, the \defemph{free quantaloid} $\Pow\cat{A}$ on $\cat{A}$ has the same objects as $\cat{A}$, and for each $A,B\in\cat{A}$, $(\Pow\cat{A})(A,B)$ is the powerset $\Pow(\cat{A}(A,B))$ equipped with the inclusion order. 
An element $f\in\cat{A}(A,B)$ is written as $f\colon A\to B$ and $\varphi\in(\Pow\cat{A})(A,B)$ as $\varphi\colon A\pto B$; we shall adopt a similar convention throughout this paper.
The composition of $\varphi\colon A\pto B$ and $\psi\colon B\pto C$ is defined as 
$\psi\circ \varphi=\{\,g\circ f\mid f\in \varphi,\, g\in \psi\,\}$.
Any morphism $f\colon A\to B$ in $\cat{A}$ gives rise to a ``singleton'' morphism $\{f\}\colon A\pto B$ in $\Pow\cat{A}$.  The identity on $A\in\Pow\cat{A}$ is $\{\id_A\}$.
Given $\varphi\colon A\pto B$, $\psi\colon B\pto C$ and $\theta\colon A\pto C$ in $\Pow\cat{A}$, we have 
\begin{equation}
\label{eqn:right-ext-right-lift-PC}
\theta\swarrow \varphi=\{\,g\in\cat{A}(B,C)\mid\forall f\in \varphi,\, g\circ f\in \theta\,\}\quad\textrm{and}\quad 
\psi\searrow \theta=\{\,f\in\cat{A}(A,B)\mid\forall g\in \psi,\, g\circ f\in \theta\,\}.\qedhere
\end{equation}
\end{example}

\begin{example}
\label{ex:QA}
Given any quantale $\cat{Q}=(Q,\leq,e,\otimes)$ and any locally small category $\cat{A}$, we can define the quantaloid $\cat{QA}$ by setting $\ob(\cat{QA})=\ob(\cat{A})$ and $(\cat{QA})(A,B)=[\cat{A}(A,B),Q]$ (the set of all functions $\cat{A}(A,B)\to Q$) equipped with the pointwise order. The composition of $\varphi\colon A\pto B$ and $\psi\colon B\pto C$ maps each $h\in\cat{A}(A,C)$ to 
\[
(\psi\circ\varphi)(h)=\bigvee\{\,\psi(g)\otimes\varphi(f)\mid f\in \cat{A}(A,B),\, g\in\cat{A}(B,C),\, g\circ f=h\,\}.
\]
A morphism $f\colon A\to B$ in $\cat{A}$ gives rise to a ``singleton'' morphism $\{f\}\colon A\pto B$ which assigns $e$ to $f$ and the least element $\bot$ of $Q$ to all morphisms $f'\in\cat{A}(A,B)$ different from $f$. The identity on $A\in\cat{QA}$ is $\{\id_A\}$.
The following slight generalisation of singleton morphisms will be used later: for each $f\colon A\to B$ in $\cat{A}$ and $\alpha\in Q$, we define the morphism $\{f\}^\alpha\colon A\pto B$ by assigning $\alpha$ to $f$ and $\bot$ to all other morphisms in $\cat{A}(A,B)$.
Given $\varphi\colon A\pto B$, $\psi\colon B\pto C$ and $\theta\colon A\pto C$ in $\cat{QA}$, we have
\begin{align*}
(\theta\swarrow\varphi)(g)&=\bigwedge\{\,\theta(g\circ f)\swarrow\varphi(f)\mid f\in\cat{A}(A,B)\,\}\quad\text{and}\\
(\psi\searrow\theta)(f)&=\bigwedge\{\,\psi(g)\searrow\theta(g\circ f)\mid g\in\cat{A}(B,C)\,\}
\end{align*}
for each $g\in \cat{A}(B,C)$ and $f\in\cat{A}(A,B)$,
where $\swarrow$ and $\searrow$ inside the curly braces denote the right extension and right lifting in $\cat{Q}$, respectively. Observe that when $\cat{Q}=\mathbf{2}$, we recover \cref{ex:PC}. 
\end{example}

\begin{example}
\label{ex:RC}
As a special case of \cref{ex:QA} with  $\cat{Q}=\Rbar$, we obtain the quantaloid $\Rbar \cat{A}$ for any locally small category $\cat{A}$. A morphism $\varphi\colon A\pto B$ in $\Rbar\cat{A}$ is a function $\varphi\colon \cat{A}(A,B)\to \Rbar$. 
The composition of $\varphi\colon A\pto B$ and $\psi\colon B\pto C$ in $\Rbar \cat{A}$ maps each $h\in\cat{A}(A,C)$ to
\[
(\psi\circ\varphi)(h)=\inf\{\,\psi(g)+\varphi(f)\mid f\in \cat{A}(A,B),\, g\in \cat{A}(B,C),\, g\circ f=h\,\}.
\] 
Given $\varphi\colon A\pto B$, $\psi\colon B\pto C$ and $\theta\colon A\pto C$ in $\Rbar\cat{A}$, we have 
\begin{align*}
(\theta\swarrow\varphi)(g)&=\sup\{\,\theta(g\circ f)-\varphi(f)\mid f\in\cat{A}(A,B)\,\}\quad\text{and}\\
(\psi\searrow\theta)(f)&=\sup\{\,\theta(g\circ f)-\psi(g)\mid g\in\cat{A}(B,C)\,\}.\qedhere
\end{align*}
\end{example}

\section{CSPs and polymorphisms via $\Pow\FinSet$}
\label{sec:polymorphism-quantaloid}
As mentioned in the Introduction, a CSP instance $I=(V,D,\cat{C})$ and its solution set  $\Sol(I)\subseteq [V,D]$ can be formulated inside the quantaloid $\Pow\FinSet$, where $\FinSet$ is the category of finite sets and functions. 
That is, we regard both $V$ and $D$ as objects of $\Pow\FinSet$, and each constraint $(k,\mathbf{x},\rho)\in\cat{C}$ as a triple consisting of the object $[k]=\{1,\dots,k\}$, the (singleton) morphism  $\{\mathbf{x}\}\colon [k]\pto V$ and the morphism $\rho\colon [k]\pto D$ in $\Pow\FinSet$.
The solution set can be regarded as a morphism $\Sol(I)\colon V\pto D$ in $\Pow\FinSet$, and may be expressed as 
$\Sol(I)=\bigcap_{(k,\mathbf{x},\rho)\in\cat{C}} \rho\swarrow\{\mathbf{x}\}$, 
in light of \eqref{eqn:right-ext-right-lift-PC}.

Let us fix a constraint language $(D,\cat{D})$. 
Observe that a CSP instance $I=(V,D,\cat{C})$ in $\CSP(\cat{D})$ can be equivalently specified by giving for each $k\in\mathbb{N}$ and each $\rho\in\cat{D}_k$, a $k$-ary relation $\sigma_\rho\subseteq V^k$ on $V$; $\sigma_\rho$ is the set of all constraint scopes $\mathbf{x}\in V^k$ such that $(k,\mathbf{x},\rho)\in\cat{C}$.
As is well-known, one can view $(V,(\sigma_\rho)_{k\in\mathbb{N},\rho\in\cat{D}_k})$ and $(D,(\rho)_{k\in\mathbb{N},\rho\in\cat{D}_k})$ as relational structures over a common relational signature, and the solutions are precisely the homomorphisms between these relational structures \cite{Feder-Vardi,Jeavons-algebraic}. 
An alternative, quantaloidal perspective is provided as follows.
For each $\rho\in\cat{D}_k$, $\sigma_\rho$ can be thought of as a (not necessarily singleton) morphism $\sigma_\rho\colon[k]\pto V$ in $\Pow\FinSet$. With this notation, the solution set is 
\begin{equation}
\label{eqn:Sol-CSP-language-based}
\Sol(I)=\bigcap_{\substack{k\in\mathbb{N}\\ \rho\in\cat{D}_k}} \rho\swarrow \sigma_\rho.
\end{equation}

Recall that to solve a CSP instance $I$ is to decide whether it has a solution or not. This amounts to deciding whether $\Sol(I)$ is empty or not. 
We can express this in the quantaloid $\Pow\FinSet$ as well. In \cref{sec:polymorphism-tropical} we shall see that formally the same construction captures the required output (the optimal value) for a certain class of optimisation problems. 
The set $[1]$ is the terminal object in $\FinSet$ (albeit not so in $\Pow\FinSet$), and thus there exists a unique function $!_D\colon D\to [1]$. This yields a canonical singleton morphism $\{!_D\}\colon D\pto [1]$ in $\Pow\FinSet$. The composition $\Opt(I)=\{!_D\}\circ\Sol(I)\colon V\pto [1]$ (which can take two values, as $\Pow\FinSet(V,[1])=\Pow(\{!_V\})=\{\emptyset, \{!_V\}\}$) is empty precisely when $\Sol(I)$ is empty, and is the singleton morphism $\{!_V\}$ otherwise. 
Therefore, to solve $I$ is to determine the morphism $\Opt(I)\colon V\pto[1]$. 

\medskip

We now move on to polymorphisms,  starting with a review of basic definitions. For any finite set $A$ and natural numbers $n$ and $k$, an \defemph{$n$-ary operation} on $A$ is a function $f\colon A^n\to A$, and a \defemph{$k$-ary relation} on $A$ is a subset $\rho\subseteq A^k$. We say that $f$ is a \defemph{polymorphism} of $\rho$ if for all $(x_{ij})\in A^{n\times k}$ we have 
\begin{equation}
\label{eqn:f-preserves-rho}
\big(
\rho(x_{11},\dots,x_{1k})\wedge \dots\wedge \rho(x_{n1},\dots,x_{nk})
\big)
\implies \rho(f(x_{11},\dots,x_{n1}),\dots,f(x_{1k},\dots,x_{nk})).
\end{equation}
We denote the set of all $n$-ary polymorphisms of $\rho$ by $\Pol(\rho)_n$. 

We express the construction  $\rho\mapsto \Pol(\rho)_n$ by means of basic operations in the quantaloid $\Pow\FinSet$.
First note that, as before, the relation $\rho\subseteq A^k$ can be seen as a morphism $\rho\colon [k]\pto A$ in $\Pow\FinSet$.
Similarly, the antecedent of \eqref{eqn:f-preserves-rho}, namely the $(n\times k)$-ary relation 
\[
\{\,(x_{ij})\in A^{n\times k}\mid\rho(x_{11},\dots,x_{1k})\wedge \dots\wedge \rho(x_{n1},\dots,x_{nk}) \,\}
\]
on $A$, can be regarded as a morphism $\rho^{\wedge n}\colon [k]\pto A^n$ in $\Pow\FinSet$.
We claim that $\rho^{\wedge n}$ is equal to the right lifting $\{\pi_i\}_{i=1}^n\searrow \rho$, where  $\{\pi_i\}_{i=1}^n\colon A^n\pto A$ is the  morphism in $\Pow\FinSet$ defined as the set of projections from the power $A^n$.
Indeed, if the tuple $(x_{ij})\in A^{n\times k}$ corresponds to the function $\chi\colon [k]\to A^n$ (so that $\chi(j)=(x_{1j},\dots,x_{nj})\in A^n$), then the function $\pi_i\circ \chi\colon [k]\to A$ corresponds to the tuple $(x_{i1},\dots,x_{ik})\in A^k$, hence the equality $\rho^{\wedge n}=\{\pi_i\}_{i=1}^n\searrow \rho$ follows from \eqref{eqn:right-ext-right-lift-PC}.
Now for each $n\in\mathbb{N}$, the set of all $n$-ary operations on $A$ preserving $\rho$ can be expressed as the right extension $\rho\swarrow \rho^{\wedge n}$, since 
\[
\rho\swarrow \rho^{\wedge n}=\{\,f\colon A^n\to A\mid\forall (x_{ij})\in\rho^{\wedge n}, (f(x_{11},\dots,x_{n1}),\dots,f(x_{1k},\dots,x_{nk}))\in\rho\,\},
\]
thus recovering the condition \eqref{eqn:f-preserves-rho}. Hence the set $\Pol(\rho)_n$ of all $n$-ary polymorphisms of $\rho$ can be written as the ``double dualisation'' $\rho\swarrow \big(\{\pi_i\}_{i=1}^n\searrow \rho\big)$ of $\{\pi_i\}_{i=1}^n$ with respect to $\rho$.
Note that for each $f\colon A^n\to A$, we have
\begin{equation}
    \label{eqn:polymorphism-searrow-rho}
    f\text{ is a polymorphism of }\rho \iff \{\pi_i\}_{i=1}^n\searrow \rho\subseteq \{f\}\searrow \rho,
\end{equation}
because 
\begin{align*}
    f\in \rho\swarrow \big(\{\pi_i\}_{i=1}^n\searrow \rho\big)
    &\iff \{f\}\subseteq \rho\swarrow \big(\{\pi_i\}_{i=1}^n\searrow \rho\big)\\
    &\iff \{f\}\circ \big(\{\pi_i\}_{i=1}^n\searrow \rho\big)\subseteq \rho\\
    &\iff
    \{\pi_i\}_{i=1}^n\searrow \rho\subseteq \{f\}\searrow\rho. 
\end{align*}

Given a family $\cat{R}=(\cat{R}_k)_{k\in\mathbb{N}}\in \prod_{k\in\mathbb{N}}\Pow(\Pow(A^k))$ of relations on $A$, we define for each $n\in\mathbb{N}$,
\[
\Pol(\cat{R})_n=\bigcap_{\substack{k\in\mathbb{N}\\
\rho\in\cat{R}_k}}\Pol(\rho)_n.
\]

\begin{proposition}
\label{prop:Pol-clone-PSet}
Let $A$ be a finite set and $\cat{R}\in\prod_{k\in\mathbb{N}}\Pow(\Pow(A^k))$.\footnote{Precisely speaking, 
the projection in clause~1 and the tupling in clause~2 must be taken with respect to the (chosen) set $\{\pi_i\}_{i=1}^n$ of projections used in the definition of $\Pol(\cat{R})$. In particular, the property of being a polymorphism for $\cat{R}$ is not invariant under composition of bijections. A similar remark applies to \cref{prop:Pol-clone-general} below.}
\begin{enumerate}
    \item For each $n\in\mathbb{N}$ and $i\in\{1,\dots,n\}$, the $i$-th projection $\pi_i\colon A^n\to A$ is in $\Pol(\cat{R})_n$.
    \item For each $m,n\in\mathbb{N}$, $g\in\Pol(\cat{R})_m$ and $f_1,\dots, f_m\in \Pol(\cat{R})_n$, we have $g\circ\langle f_1,\dots, f_m\rangle\in \Pol(\cat{R})_n$. (Here, $\langle f_1,\dots,f_m\rangle\colon A^n\to A^m$ is the tupling of $f_1,\dots,f_m\colon A^n\to A$.)
\end{enumerate}
\end{proposition}
\begin{proof}
One can easily show these using \eqref{eqn:f-preserves-rho}.
We shall present an alternative proof making full use of quantaloidal structure (\eqref{eqn:adjointness-quantaloid} and \cref{prop:quantaloid-basic-equations}), as a warm-up for a similar proof of \cref{prop:Pol-clone-general}.

Clearly, it suffices to consider the case where $\cat{R}$ consists of a single (say, $k$-ary) relation $\rho$. We shall use the criterion \eqref{eqn:polymorphism-searrow-rho}.
Clause 1 is clear because $(-)\searrow\rho$ is order-reversing. 
Under the assumption of clause 2, we have $\{\pi_i\}_{i=1}^m\searrow \rho\subseteq \{g\}\searrow\rho$ and $   \{\pi_i\}_{i=1}^n\searrow \rho\subseteq \{f_1,\dots,f_m\}\searrow\rho$. Hence, 
\begin{align*}
    \{\pi_i\}_{i=1}^n\searrow \rho
    &\subseteq \{f_1,\dots,f_m\}\searrow\rho\\
    &= \big(\{\pi_i\}_{i=1}^m\circ\{\langle f_1,\dots,f_m\rangle\}\big)\searrow\rho\\
    &=\{\langle f_1,\dots,f_m\rangle\}\searrow\big(\{\pi_i\}_{i=1}^m\searrow\rho\big)\\
    &\subseteq \{\langle f_1,\dots,f_m\rangle\}\searrow\big(\{g\}\searrow\rho\big)\\
    &=
    \big(\{g\}\circ\{\langle f_1,\dots,f_m\rangle\}\big)\searrow\rho\\
    &=\{g\circ\langle f_1,\dots,f_m\rangle\}\searrow\rho.\qedhere
\end{align*}
\end{proof}

\cref{prop:Pol-clone-PSet} states that polymorphisms form a clone: recall that a  \defemph{(concrete) clone} on a set $A$ is a family $\cat{F}\in \prod_{n\in\mathbb{N}}\Pow([A^n,A])$ of operations on $A$ satisfying the following.
\begin{enumerate}
    \item For each $n\in\mathbb{N}$ and $i\in\{1,\dots,n\}$, the $i$-th projection $\pi_i\colon A^n\to A$ is in $\cat{F}_n$.
    \item For each $m,n\in\mathbb{N}$, $g\in\cat{F}_m$ and $f_1,\dots, f_m\in \cat{F}_n$, we have $g\circ\langle f_1,\dots, f_m\rangle\in \cat{F}_n$.\qedhere
\end{enumerate}

Let us now formalise the informal claim that if a problem has certain symmetry, then so does its solution.

\begin{proposition}
\label{prop:Pol-closure-properties}
Let $A$ be a finite set. 
\begin{enumerate}
    \item If $k\in\mathbb{N}$ and $(\rho_j\colon [k]\pto A)_{j\in J}$ is a family of $k$-ary relations on $A$, then for each $n\in\mathbb{N}$ we have  $\bigcap_{j\in J}\Pol(\rho_j)_n\subseteq \Pol(\bigcap_{j\in J}\rho_j)_n$.
    \item If $k,l\in\mathbb{N}$, $\rho\colon [k]\pto A$ is a $k$-ary relation on $A$, and $\sigma\colon [k]\pto[l]$ is a morphism in $\Pow\FinSet$, then for each $n\in\mathbb{N}$ we have  $\Pol(\rho)_n\subseteq \Pol(\rho\swarrow \sigma)_n$.
\end{enumerate}
\end{proposition}
\begin{proof}
Recall that the condition for an $n$-ary operation $f$ to be a polymorphism can be expressed as \eqref{eqn:polymorphism-searrow-rho}.

\begin{enumerate}
    \item If $f\in \bigcap_{j\in J}\Pol(\rho_{j})_n$, then we have $\{\pi_i\}_{i=1}^n\searrow\rho_{j}\subseteq\{f\}\searrow\rho_{j}$ for each $j\in J$.  By  \cref{prop:quantaloid-basic-equations},
\[
\{\pi_i\}_{i=1}^n\searrow\bigcap_{j\in J}\rho_{j}=\bigcap_{j\in J}(\{\pi_i\}_{i=1}^n\searrow\rho_{j})\subseteq \bigcap_{j\in J}(\{f\}\searrow\rho_{j})=\{f\}\searrow \bigcap_{j\in J}\rho_{j}.
\]
\item If $f\in\Pol(\rho)_n$, then we have $\{\pi_i\}_{i=1}^n\searrow\rho\subseteq\{f\}\searrow\rho$. By \cref{prop:quantaloid-basic-equations},
\[
\{\pi_i\}_{i=1}^n\searrow (\rho\swarrow \sigma)=(\{\pi_i\}_{i=1}^n\searrow \rho)\swarrow \sigma\subseteq (\{f\}\searrow \rho)\swarrow \sigma=\{f\}\searrow(\rho\swarrow\sigma).\qedhere
\]
\end{enumerate}
\end{proof}

In view of \eqref{eqn:Sol-CSP-language-based}, we immediately have the following.
\begin{corollary}
\label{cor:Pol-Sol-CSP}
Let $D$ be a finite set, $\cat{D}\in \prod_{k\in\mathbb{N}}\Pow(\Pow(D^k))$ and $I=(V,D,\cat{C})\in\CSP(\cat{D})$. Then for each $n\in\mathbb{N}$ we have $\Pol(\cat{D})_n\subseteq \Pol(\Sol(I))_n$.
\end{corollary}

We conclude this section with a precise statement of the dichotomy theorem. There are several possible ways to phrase the theorem, the following being one of them (see~\cite[Theorem 41]{BKW17}).
\begin{definition}
\label{def:Siggers}
A $4$-ary operation $f : D^4 \to D$ on a finite set $D$ is said to be \defemph{Siggers}
if it satisfies
\begin{align*}
    f(y, x, y, z) = f(x, y ,z, x)
\end{align*}
for all $(x, y ,z) \in D^3$.
\end{definition}
\begin{theorem}[{Dichotomy theorem~\cite{Bul17,Zhu20}}]
\label{thm:dichotomy-CSP}
Let $D$ be a finite set and $\cat{D}$ a finite set of relations on $D$.
If some Siggers operation is a polymorphism of $\cat{D}$,
then $\CSP(\cat{D})$ is in P.
Otherwise, it is NP-complete.
\end{theorem}

\section{The quantaloidal CSP}
\label{sec:CSP-in-QA}
The above reformulation of CSPs and polymorphisms in the quantaloid $\Pow\FinSet$ suggests that a certain part of the mathematical theory of the CSP can be developed in a much broader context. 
In this section we shall embark on such a development. Although the discussion below might look rather formal, we shall apply it to a certain class of optimisation problems in the next section (which can be read independently of this section).
We remark that in the special case where $\cat{Q}=\mathbf{2}$ (hence $\cat{QA}$ is the free quantaloid $\cat{PA}$), some of the notions introduced below appear in \cite{Kerkhoff-general}.

Let us take any quantale $\cat{Q}=(Q,\leq,e,\otimes)$ and any locally small category $\cat{A}$ with finite products; we shall work within the quantaloid $\cat{QA}$ instead of $\Pow\FinSet$.
For objects $A,K\in\cat{A}$, we define a \defemph{$K$-ary $\cat{Q}$-valued relation} on $A$ to be a morphism $K\pto A$ in $\cat{QA}$. 
Given such a $\cat{Q}$-valued relation $\rho\colon K\pto A$ and a natural number $n$, we define the (totality of) \defemph{$n$-ary $\cat{Q}$-valued polymorphisms} for $\rho$ as $\Pol(\rho)_n=\rho\swarrow\big(\{\pi_i\}_{i=1}^n\searrow \rho\big)$, where $\{\pi_i\}_{i=1}^n\colon A^n\pto A$ is the morphism in $\cat{QA}$ defined as 
\[
\{\pi_i\}_{i=1}^n(f)=\begin{cases}
e&\text{if $f$ is the $i$-th projection $\pi_i\colon A^n\to A$ for some $i\in\{1,\dots,n\}$},\\
\text{the least element $\bot$ of  $Q$}&\text{otherwise},
\end{cases}
\]
for all $f\in\cat{A}(A^n,A)$. 
Notice that $\Pol(\rho)_n$ is a morphism $A^n\pto A$ in $\cat{QA}$, i.e., it assigns to each morphism $f\colon A^n\to A$ in $\cat{A}$ an element $\Pol(\rho)_n(f)$ of $Q$, which may be thought of as the degree to which $f$ is a polymorphism of $\rho$.

\emph{Individual} polymorphisms (as opposed to the \emph{totality} of them) can then be defined as follows. An  (individual) \defemph{$n$-ary $\cat{Q}$-valued polymorphism} for $\rho\colon K\pto A$ is a pair $(f\in\cat{A}(A^n,A), \alpha\in Q)$ such that $\alpha\leq \Pol(\rho)_n(f)$. Using a notation introduced in \cref{ex:QA}, the latter condition is equivalent to $\{\pi_i\}_{i=1}^n\searrow \rho\leq \{f\}^\alpha\searrow\rho$; cf.~\eqref{eqn:polymorphism-searrow-rho}.
This in turn amounts to the following more explicit condition, generalising \eqref{eqn:f-preserves-rho}: for any $\chi\colon K\to A^n$ in $\cat{A}$, we have 
$\alpha\otimes \big(\rho(\pi_1\circ\chi)\wedge\dots\wedge \rho(\pi_n\circ \chi)\big)\leq \rho(f\circ \chi)$.

For a set $\cat{R}$ of $\cat{Q}$-valued relations on $A\in\cat{A}$  (i.e., $\cat{R}$ is a set of morphisms in $\cat{QA}$ with codomain $A$) and  $n\in\mathbb{N}$, we define $\Pol(\cat{R})_n=\bigwedge_{\rho\in \cat{R}}\Pol(\rho)_n$. We say $(f\in\cat{A}(A^n,A),\alpha\in Q)$ is an \defemph{$n$-ary $\cat{Q}$-valued polymorphism} of $\cat{R}$ if $\alpha\leq \Pol(\cat{R})_n(f)$. 

The following propositions generalise \cref{prop:Pol-clone-PSet,prop:Pol-closure-properties}, respectively.
\begin{proposition}
\label{prop:Pol-clone-general}
Let $\cat{Q}=(Q,\leq,e,\otimes)$ be a quantale, $\cat{A}$ a locally small category with finite products, $A\in\cat{A}$ and $\cat{R}$ a set of $\cat{Q}$-valued relations on $A$. 
\begin{enumerate}
    \item For each $n\in\mathbb{N}$ and $i\in\{1,\dots,n\}$, the $i$-th projection $\pi_i\colon A^n\to A$ satisfies $e\leq \Pol(\cat{R})_n(\pi_i)$.
    \item For each $m,n\in\mathbb{N}$, $g\colon A^m\to A$ and $f_1,\dots, f_m\colon A^n\to A$ in $\cat{A}$, we have 
    \begin{equation*}
    \Pol(\cat{R})_m(g)\otimes\big(\Pol(\cat{R})_n(f_1)\wedge \dots\wedge \Pol(\cat{R})_n(f_m) \big)\leq \Pol(\cat{R})_n(g\circ\langle f_1,\dots,f_m\rangle).
    \end{equation*}
\end{enumerate}
\end{proposition}
\begin{proof}
We consider the case where $\cat{R}$ consists of a single morphism $\rho\colon K\pto A$ in $\cat{QA}$.
Clause~1 amounts to the claim that $\{\pi_i\}\leq \rho\swarrow \big(\{\pi_i\}_{i=1}^n\searrow \rho\big)$, and can be shown as in the proof of \cref{prop:Pol-clone-PSet}. 
For clause~2, it suffices to show that for any $\alpha,\beta\in Q$, $\beta\leq \Pol(\rho)_m(g)$ and $\alpha\leq\Pol(\rho)_n(f_1)\wedge\dots\wedge \Pol(\rho)_n(f_m)$ imply $\beta\otimes\alpha\leq \Pol(\rho)_n(g\circ\langle f_1,\dots,f_m\rangle)$. 
Define $\{f_1,\dots,f_m\}^\alpha\colon A^n\pto A$ as the join of $\{f_1\}^\alpha,\dots,\{f_m\}^\alpha$ (cf.~\cref{ex:QA}); that is,
\[
\{f_1,\dots,f_m\}^\alpha(f')=\begin{cases}
\alpha&\text{if $f'=f_i$ for some $i\in\{1,\dots,m\}$},\\
\bot
&\text{otherwise}
\end{cases}
\]
for all $f'\in\cat{A}(A^n,A)$.
The assumptions amount to $\{g\}^\beta\leq \rho\swarrow\big(\{\pi_i\}_{i=1}^m\searrow \rho\big)$ and $\{f_1,\dots,f_m\}^\alpha\leq \rho\swarrow\big(\{\pi_i\}_{i=1}^n\searrow \rho\big)$, i.e., $\{\pi_i\}_{i=1}^m\searrow \rho\leq \{g\}^\beta\searrow \rho$ and $\{\pi_i\}_{i=1}^n\searrow \rho\leq \{f_1,\dots,f_m\}^\alpha\searrow \rho$.
Hence
\begin{align*}
    \{\pi_i\}_{i=1}^n\searrow \rho
    &\leq \{f_1,\dots,f_m\}^\alpha\searrow\rho\\
    &= \big(\{\pi_i\}_{i=1}^m\circ\{\langle f_1,\dots,f_m\rangle\}^\alpha\big)\searrow\rho\\
    &=\{\langle f_1,\dots,f_m\rangle\}^\alpha\searrow\big(\{\pi_i\}_{i=1}^m\searrow\rho\big)\\
    &\leq \{\langle f_1,\dots,f_m\rangle\}^\alpha\searrow\big(\{g\}^\beta\searrow\rho\big)\\
    &=
    \big(\{g\}^\beta\circ\{\langle f_1,\dots,f_m\rangle\}^\alpha\big)\searrow\rho\\
    &=\{g\circ\langle f_1,\dots,f_m\rangle\}^{\beta\otimes\alpha}\searrow\rho,
\end{align*}
showing $\beta\otimes\alpha\leq \Pol(\rho)_n(g\circ\langle f_1,\dots,f_m\rangle)$.
\end{proof}

\begin{proposition}
\label{prop:Pol-closure-properties-QA}
Let $\cat{Q}$ be a quantale, $\cat{A}$ a locally small category with finite products and $A\in\cat{A}$. 
\begin{enumerate}
    \item If $K\in\cat{A}$ and $(\rho_j\colon K\pto A)_{j\in J}$ is a family of $K$-ary $\cat{Q}$-valued relations on $A$, then for each $n\in\mathbb{N}$, we have $\bigwedge_{j\in J}\Pol(\rho_j)_n\leq \Pol(\bigwedge_{j\in J}\rho_j)_n$.
    \item If $K,L\in\cat{A}$, $\rho\colon K\pto A$ is a $K$-ary $\cat{Q}$-valued relation on $A$, and $\sigma\colon K\pto L$ is a morphism in $\cat{QA}$, then for each $n\in\mathbb{N}$, we have $\Pol(\rho)_n\leq \Pol(\rho\swarrow \sigma)_n$.
\end{enumerate}
\end{proposition}
\begin{proof}
One can show these by a straightforward modification of the proof of \cref{prop:Pol-closure-properties}, along the lines of the proof of \cref{prop:Pol-clone-general}.
\end{proof}

In the remainder of this section, we shall sketch the \emph{quantaloidal CSP} 
in $\cat{QA}$. 
In order to render the following as a well-defined computational problem, we would have to specify suitable machine representations of the data involved. We shall omit such considerations in this section, and simplify the discussion by allowing infinitely many constraints as well as infinite constraint languages.

An \emph{instance} $I=(V,D,\cat{C})$ consists of objects $V,D\in\cat{A}$ and a set $\cat{C}$ of \emph{$\cat{Q}$-valued constraints (in $\cat{A}$)}. 
There seems to be a few possibilities concerning the detail of a definition of $\cat{Q}$-valued constraint.

\begin{enumerate}
\item A straightforward approach is to define a $\cat{Q}$-valued constraint as a triple $(K,\mathbf{x},\rho)$ consisting of an object $K\in\cat{A}$, a morphism $\mathbf{x}\colon K\to V$ in $\cat{A}$ and a morphism $\rho\colon K\pto D$ in $\cat{QA}$. We may then define the morphism $\Sol(I)\colon V\pto D$ by 
$\Sol(I)=\bigwedge_{(K,\mathbf{x},\rho)\in \cat{C}} \rho\swarrow \{\mathbf{x}\}$.
\item The second possibility is to define a {$\cat{Q}$-valued constraint} as a quadruple $(K,\mathbf{x},\alpha,\rho)$, adding a new component $\alpha\in Q$.
$\Sol(I)$ is now given as $\bigwedge_{(K,\mathbf{x},\alpha,\rho)\in \cat{C}} \rho\swarrow \{\mathbf{x}\}^\alpha$.
This latter formulation seems to be better suited for considerations involving $\cat{Q}$-valued constraint languages.
A \emph{$\cat{Q}$-valued constraint language} consists of a pair $(D,\cat{D})$ of an object $D\in\cat{A}$ and a set $\cat{D}$ of morphisms in $\cat{QA}$ with codomain $D$.
Given an instance $I=(V,D,\cat{C})$ such that $\rho\in\cat{D}$ for each $(K,\mathbf{x},\alpha,\rho)\in\cat{C}$, we may define for each (say, $K$-ary) $\rho\in \cat{D}$ the morphism $\sigma_\rho\colon K\pto V$ as the supremum of all morphisms $\{\mathbf{x}\}^\alpha$ with $(K,\mathbf{x},\alpha,\rho)\in\cat{C}$.
In view of the equation $\bigwedge_{j\in J}\big(\rho\swarrow \sigma_j\big)=\rho\swarrow \big(\bigvee_{j\in J}\sigma_j\big)$, we have 
$\Sol(I)=\bigwedge_{\rho\in \cat{D}} \rho\swarrow \sigma_\rho$. Notice that \cref{prop:Pol-closure-properties-QA} implies that any $\cat{Q}$-valued polymorphism for $\cat{D}$ is a $\cat{Q}$-valued  polymorphism for $\Sol(I)$.
\item This suggests the third, most general definition of $\cat{Q}$-valued constraint; it is a triple $(K,\sigma,\rho)$ consisting of an object $K\in\cat{A}$ and morphisms $\sigma\colon K\pto V$ and $\rho\colon K\pto D$ in $\cat{QA}$. We now have $\Sol(I)=\bigwedge_{(K,\sigma,\rho)\in\cat{C}}\rho\swarrow\sigma$. In the next section, we shall adopt (a finitary version of) this definition.
\end{enumerate}
In each case, the goal is to determine the value $\Opt(I)=\{!_D\}\circ\Sol(I)\colon V\pto 1$, where $1$ is the terminal object of $\cat{A}$. Notice that since $(\cat{QA})(V,1)=[\cat{A}(V,1),Q]\cong Q$, we can naturally identify $\Opt(I)$ with an element of $Q$, the ``optimal value'' of $I$.

\section{Quantaloidal CSPs in $\Rbar\FinSet$ and $\Rbar\Set$ as optimisation problems}
\label{sec:polymorphism-tropical}

In this section, 
we consider a certain class of optimisation problems which we call the \emph{tropical valued CSP} (TVCSP). The TVCSP is a subclass of the quantaloidal CSP in the quantaloid $\Rbar\Set$.  
A \defemph{TVCSP instance} $I$
consists of a \emph{finite} set $V$ of variables, a (possibly infinite) set $D$ called the domain, 
and a \emph{finite} set $\cat{C}$ of $\Rbar$-valued constraints.
Here, we define an \defemph{$\Rbar$-valued constraint} as a triple $(k,\sigma,\rho)$ consisting of a natural number $k$
and
morphisms $\sigma : [k] \pto V$
and $\rho : [k] \pto D$
in $\Rbar \Set$.
For a TVCSP instance $I = (V, D, \cat{C})$,
the morphism $\Sol(I) : V \pto D$ in $\Rbar\Set$ maps each $s : V \to D$ to
\begin{align}
\label{eqn:Sol-TVCSP}
   \Sol(I)(s) = \sup_{(k, \sigma, \rho) \in \cat{C}} \sup_{\mathbf{x} \in V^k} \left(\rho(s(\mathbf{x})) - \sigma(\mathbf{x})\right),
\end{align}
where $s(\mathbf{x})$ denotes the composite $s\circ\mathbf{x}\colon [k]\to D$.
To solve the TVCSP instance $I$ is to compute
\begin{align*}
    \Opt(I) = \inf_{s : V \to D} \Sol(I)(s).
\end{align*}

Thus the TVCSP is a problem of computing a minimax value, and it can model scenarios in which we wish to ``minimise the maximum loss'' or ``optimise the worst case''.

\begin{example}
Consider a scheduling problem, in which we are given multiple \emph{activities} $1,\dots,n$, \emph{precedence relations} among the activities of the form ``activity $j$ cannot start until activity $i$ finishes'', 
the \emph{processing time} $p_i\in\mathbb{N}$ of each activity $i$ (so that if activity $i$ starts at time $s(i)\in\mathbb{N}$, then it finishes at time $s(i)+p_i\in\mathbb{N}$) and the \emph{due date} $d_i\in\mathbb{N}$ of each activity $i$.
We are interested in a schedule of the activities (a function $s\colon [n]\to \mathbb{N}$) that minimises the maximum deviation from due dates ($\max_{i\in[n]}|d_i-(s(i)+p_i)|$). 
We can model this as a TVCSP instance by setting $V=[n]$ and $D=\mathbb{N}$ (or $D=[N]$ for a suitably large $N\in\mathbb{N}$), expressing the maximum deviation by an $\Rbar$-valued relation, and encoding the precedence relations (and the processing time) using $\Rbar$-valued constraints taking values in $\{0,\infty\}$. 
\end{example}

We explain our choice of the name ``tropical valued CSP''.
The \emph{valued CSP} (VCSP) 
is a well-known optimisation variant of the CSP (see, e.g.,~\cite{Ziv12}).
The data of a VCSP instance $I=(V,D,\cat{C})$ is similar to that of a TVCSP instance, 
and the goal is to compute the infimum of
\begin{equation}
\label{eqn:VCSP-objective}
\sum_{(k,\sigma,\rho)\in\cat{C}}
\sum_{\mathbf{x} \in V^k}\sigma(\mathbf{x}) \cdot  \rho(s(\mathbf{x}))
\end{equation}
over $s : V \to D$.
(Precisely, we have to assume, e.g., that $0 \leq \sigma(\mathbf{x}) < \infty$ for each $\mathbf{x} \in V^k$ in the VCSP.)
We can regard \eqref{eqn:Sol-TVCSP} as a variant of \eqref{eqn:VCSP-objective}, in which \emph{addition} is replaced by \emph{supremum} and \emph{multiplication} by \emph{subtraction}. This is analogous to the transition from the field $\mathbb{R}$ of real numbers to the quantale $\Rbar$, which may be thought of as a variant of the \emph{tropical semiring} (see, e.g., \cite{Speyer-Sturmfels}); roughly, the latter is obtained from the former by replacing \emph{addition} by \emph{infimum} and \emph{multiplication} by \emph{addition}.\footnote{It is known that computational complexity of VCSPs can be captured by the notion of \emph{weighted polymorphism}~\cite{CCC13}. 
Weighted polymorphisms differ substantially from our $\Rbar$-valued polymorphisms, and we have not been able to understand the former from a categorical perspective.}

\subsection{The dichotomy theorem for TVCSPs with finite domains}
\label{subsec:TVCSP-FinSet}
In this subsection,
we consider TVCSPs in which the domains $D$ are also finite.
Thus we may think of our problems as a subclass of the quantaloidal CSP in the quantaloid $\Rbar \FinSet$.

We shall classify TVCSPs in terms of computational complexity.
In order to do so rigorously,
we specify a representation of instances as follows.
We assume that $\sigma$ and $\rho$ in each $\Rbar$-valued constraint $(k,\sigma,\rho)$ take values in $\mathbb{Q}\cup\{\pm\infty\}$.
Furthermore, $(k, \sigma, \rho)$
is given by the lists of all pairs $(\mathbf{x}, \sigma(\mathbf{x}))$ for $\mathbf{x} \in \dom \sigma$
and of all pairs $(\mathbf{d}, \rho(\mathbf{d}))$ for $\mathbf{d} \in \dom \rho$,
where
$\dom \sigma = \{\, \mathbf{x} \in V^k \mid \sigma(\mathbf{x}) < \infty \,\}$
and $\dom \rho = \{\, \mathbf{d} \in D^k \mid \rho(\mathbf{d}) < \infty \,\}$.
Hence the input size of an instance $I = (V, D, \cat{C})$
is $O(|V| + |D| + \sum_{(k, \sigma, \rho) \in \cat{C}} (|\dom \sigma| + |\dom \rho|))$.
Let $\cat{D}$ be a finite set of $\Rbar$-valued relations on a finite set $D$.
$\TVCSP(\cat{D})$ denotes the class of all TVCSP instances $I = (V, D', \cat{C})$
such that $D' = D$ and, for each $\Rbar$-valued constraint $(k, \sigma, \rho) \in \cat{C}$, we have $\rho \in \cat{D}$.

For a $k$-ary $\Rbar$-valued relation $\rho$ and $\alpha\in \Rbar$ with $\alpha < \infty$, 
we denote by ${\rho}^{\alpha}$ the sublevel set of $\rho$ with respect to $\alpha$, i.e.,
${\rho}^\alpha = \{\, \mathbf{d} \in D^k \mid \alpha\geq\rho(\mathbf{d}) \,\}$.
We define $U\cat{D} = \{\, {\rho}^\alpha \mid  \rho \in \cat{D},\, \alpha < \infty \,\}$.

\begin{proposition}
\label{prop:dichotomy-TVCSP}
Let $D$ be a finite set and $\cat{D}$ a finite set of $\Rbar$-valued relations on $D$.
Then $\TVCSP(\cat{D})$ and $\CSP(U\cat{D})$ are polynomial-time reducible to each other.
\end{proposition}
\begin{proof}
We first give a polynomial-time reduction from $\TVCSP(\cat{D})$ to $\CSP(U\cat{D})$.
Take an arbitrary TVCSP instance $I = (V, D, \cat{C}) \in \TVCSP(\cat{D})$.
For any $\alpha < \infty$,
we obtain
\begin{align*}
    \alpha\geq\Opt(I) &\iff \exists s : V \to D.~\forall (k, \sigma, \rho) \in \cat{C}.~\forall \mathbf{x} \in V^k.~\alpha \geq \rho(s(\mathbf{x})) - \sigma(\mathbf{x})\\
    &\iff \exists s : V \to D.~\forall (k, \sigma, \rho) \in \cat{C}.~\forall \mathbf{x} \in V^k.~\sigma(\mathbf{x}) + \alpha \geq \rho(s(\mathbf{x}))\\
    &\iff \exists s : V \to D.~\forall (k, \sigma, \rho) \in \cat{C}.~\forall \mathbf{x} \in \dom \sigma.~\sigma(\mathbf{x}) + \alpha \geq \rho(s(\mathbf{x}))\\
    &\iff \exists s : V \to D.~\forall (k, \sigma, \rho) \in \cat{C}.~\forall \mathbf{x} \in \dom \sigma.~s(\mathbf{x}) \in \rho^{\sigma(\mathbf{x}) + \alpha}.
\end{align*}
Define
the CSP instance $I^\alpha = (V, D, \cat{C}^\alpha)$
by
\begin{align*}
    \cat{C}^\alpha = \{\, (k, \mathbf{x}, {\rho}^{\sigma(\mathbf{x}) + \alpha}) \mid (k, \sigma, \rho) \in \cat{C},\, \mathbf{x} \in \dom \sigma \,\}.
\end{align*}
Then we have
\begin{align}
\label{eq:Opt iff Sol}
\alpha\geq\Opt(I) \iff \Sol(I^\alpha) \neq \emptyset.
\end{align}
Note that,
since $\alpha < \infty$,
$I^\alpha \in \CSP(U\cat{D})$ and the input size of $I^\alpha$ is bounded by a polynomial in that of $I$.
(Indeed, such a bound is provided by $|V| + |D| + \sum_{(k, \sigma, \rho) \in \cat{C}} |\dom \sigma| |\dom \rho|$.)
The relation~\eqref{eq:Opt iff Sol}
says that
the computation of $\Opt(I)$ can be reduced to the problem of finding the minimum $\alpha$ such that $\Sol(I^\alpha) \neq \emptyset$.
Since
\begin{align*}
    \Opt(I)
    \in \{\, \rho(\mathbf{d}) - \sigma(\mathbf{x}) \mid (k, \sigma, \rho) \in \cat{C},\, \mathbf{d} \in \dom \rho,\, \mathbf{x} \in \dom \sigma \,\} \cup \{\pm\infty\},
\end{align*}
it suffices to solve the CSP instance $I^\alpha$ only for
\begin{align*}
    \alpha \in \left(\{\, \rho(\mathbf{d}) - \sigma(\mathbf{x}) \mid (k, \sigma, \rho) \in \cat{C},\, \mathbf{d} \in \dom \rho,\, \mathbf{x} \in \dom \sigma \,\} \setminus \{\infty\}\right) \cup \{-\infty\}.
\end{align*}
Since $|\{\, \rho(\mathbf{d}) - \sigma(\mathbf{x}) \mid (k, \sigma, \rho) \in \cat{C}, \mathbf{d} \in \dom \rho, \mathbf{x} \in \dom \sigma \,\}| \leq \sum_{(k, \sigma, \rho) \in \cat{C}} |\dom \sigma| |\dom \rho|$, 
we can compute $\Opt(I)$ by solving polynomially many instances in $\CSP(U \cat{C})$.
Thus $\TVCSP(\cat{D})$ is polynomial-time reducible to $\CSP(U\cat{D})$.

We then give a polynomial-time reduction from $\CSP(U\cat{D})$ to $\TVCSP(\cat{D})$.
For $\alpha < \infty$ and $\mathbf{x} \in V^k$,
recall the morphism $\{\mathbf{x}\}^\alpha : [k] \pto V$ in $\Rbar\FinSet$
defined in \cref{ex:QA} as
\begin{align*}
    \{\mathbf{x}\}^\alpha(\mathbf{x}') =
    \begin{cases}
    \alpha & \text{if $\mathbf{x}' = \mathbf{x}$},\\
    \infty & \text{otherwise}.
    \end{cases}
\end{align*}
From a CSP instance $I_0 = (V, D, \cat{C}_0) \in \CSP(U\cat{D})$,
we construct a TVCSP instance $I = (V, D, \cat{C})$ where
$
    \cat{C} = \{\, (k, \{\mathbf{x}\}^\alpha, \rho) \mid (k, \mathbf{x}, \rho^\alpha) \in \cat{C}_0 \,\}$.
(Note that for each $(k,\mathbf{x},\rho')\in\cat{C}_0$, there exist $\rho\in\cat{D}_k$ and $\alpha<\infty$ such that $\rho'=\rho^\alpha$, and we can find such a pair $(\rho,\alpha)$ in polynomial time by inspecting all $\rho \in \cat{D}_k$ and $\alpha \in \{ \rho(\mathbf{x}) \mid \mathbf{x} \in \dom \rho \}$.)
Then
\begin{align*}
    \Sol(I_0) \neq \emptyset
    &\iff \exists s : V \to D.~\forall (k, \mathbf{x}, \rho^\alpha) \in \cat{C}_0.~\alpha \geq \rho(s(\mathbf{x}))\\
    &\iff \exists s : V \to D.~\forall (k, \mathbf{x}, \rho^\alpha) \in \cat{C}_0.~0 \geq \rho(s(\mathbf{x})) - \alpha\\
    &\iff \exists s : V \to D.~\forall (k, \{\mathbf{x}\}^\alpha, \rho) \in \cat{C}.~\forall \mathbf{x}' \in \dom \{\mathbf{x}\}^\alpha.~0 \geq \rho(s(\mathbf{x}')) - \{\mathbf{x}\}^\alpha(\mathbf{x}')\\
    &\iff \exists s : V \to D.~\forall (k, \{\mathbf{x}\}^\alpha, \rho) \in \cat{C}.~\forall \mathbf{x}' \in V^k.~0 \geq \rho(s(\mathbf{x}')) - \{\mathbf{x}\}^\alpha(\mathbf{x}')\\
    &\iff 0 \geq \Opt(I).
\end{align*}
Thus we can solve $I_0$ by determining if $0 \geq \Opt(I)$.
This gives a polynomial-time reduction from $\CSP(U\cat{D})$ to $\TVCSP(\cat{D})$.
\end{proof}

Hence the classification of TVCSPs is reduced to that of CSPs. In particular, the dichotomy theorem for CSPs (\cref{thm:dichotomy-CSP}) implies the dichotomy for TVCSPs: $\TVCSP(\cat{D})$ is either in P or NP-hard.
We note that a relation analogous to that between $\TVCSP(\cat{D})$
and $\CSP(U\cat{D})$ described in \cref{prop:dichotomy-TVCSP} has already been observed in the context of the \emph{fuzzy CSP}~\cite[Section 9.4.3]{RvW06}, which can be seen as a special case of the TVCSP
{(see also \cite{HMV17})}.
However, the situation for the TVCSP is subtler due to the coexistence of $\infty$ and $-\infty$, and our proof of \cref{prop:dichotomy-TVCSP} relies heavily on the adjointness relation \eqref{eqn:adjointness-quantaloid} in $\Rbar$ as well as the details of the operations $+$ and $-$ (\cref{table:Rbar}). 

The above dichotomy for TVCSPs can be captured by a suitable notion of polymorphism.
Specialising the notion of $\cat{Q}$-valued polymorphism in \cref{sec:CSP-in-QA} to the quantaloid $\Rbar\FinSet$, we define an ($n$-ary) \defemph{$\Rbar$-valued polymorphism} of an $\Rbar$-valued relation  $\rho\colon [k]\pto A$ on a finite set $A$ to be a pair $(f,\alpha)$ of a function $f\colon A^n\to A$ and $\alpha\in \Rbar$ such that for all $(x_{ij})\in A^{n\times k}$, we have 
\begin{align*}
    \alpha + \sup\{\, \rho(x_{11},\dots, x_{1k}), \dots, \rho(x_{n1},\dots,x_{nk}) \,\} \geq \rho(f(x_{11},\dots,x_{n1}), \dots, f(x_{1k},\dots,x_{nk})).
\end{align*}
Given a set $\cat{R}$ of $\Rbar$-valued relations on $A$, we say $(f,\alpha)$ is an \defemph{$\Rbar$-valued polymorphism} of $\cat{R}$ if it is so for every element of $\cat{R}$.

\begin{lemma}
\label{lem:f-zero}
Let $A$ be a finite set and $\cat{R}$ a set of $\Rbar$-valued relations on $A$. 
For any function $f \colon A^n \to A$,
$(f,0)$ is an $\Rbar$-valued polymorphism of $\cat{R}$
if and only if
$f$ is a polymorphism of $U\cat{R}$.
\end{lemma}
\begin{proof}
We may assume that $\cat{R}$ consists of a single (say, $k$-ary) $\Rbar$-valued relation $\rho\colon [k]\pto A$.

First suppose that $(f,0)$ is an $\Rbar$-valued polymorphism of $\rho$. Our aim is to show that for every $\alpha <\infty$, $f$ is a polymorphism of ${\rho}^{\alpha}$, i.e., that for any $(x_{ij})\in A^{n\times k}$ we have
\begin{equation}
\label{eqn:rho-alpha-polymorphism}
\big(
\rho^\alpha(x_{11},\dots,x_{1k})\wedge \dots\wedge \rho^\alpha(x_{n1},\dots,x_{nk})
\big)
\implies \rho^\alpha(f(x_{11},\dots,x_{n1}),\dots,f(x_{1k},\dots,x_{nk})).
\end{equation}
Assume the antecedent of \eqref{eqn:rho-alpha-polymorphism}, i.e., that  $\alpha\geq \rho(x_{i1},\dots,x_{ik})$ for every $i\in\{1,\dots, n\}$.
Since $(f,0)$ is an $\Rbar$-valued polymorphism of $\rho$, we have 
\[
\sup\{\,\rho(x_{11},\dots,x_{1k}), \dots, \rho(x_{n1},\dots,x_{nk})\,\}\geq \rho(f(x_{11},\dots,x_{n1}),\dots,f(x_{1k},\dots,x_{nk})).
\]
Thus it follows that $\alpha\geq\rho(f(x_{11},\dots,x_{n1}),\dots,f(x_{1k},\dots,x_{nk}))$, showing the consequent of \eqref{eqn:rho-alpha-polymorphism}. 

Next suppose that $(f,0)$ is not an $\Rbar$-valued polymorphism of $\rho$. This means that there exists a tuple $(x_{ij})\in A^{n\times k}$ such that 
\[
\sup\{\,\rho(x_{11},\dots,x_{1k}), \dots, \rho(x_{n1},\dots,x_{nk})\,\}< \rho(f(x_{11},\dots,x_{n1}),\dots,f(x_{1k},\dots,x_{nk})).
\]
Let $\alpha$ be the value of the left-hand side. Then $\alpha<\infty$ and $f$ is not a polymorphism of $\rho^\alpha$; indeed, our choice of $(x_{ij})$ and $\alpha$ implies that \eqref{eqn:rho-alpha-polymorphism} is violated. 
\end{proof}

As an easy consequence of \cref{thm:dichotomy-CSP}, \cref{prop:dichotomy-TVCSP} and \cref{lem:f-zero}, we have the following criterion of computational complexity in terms of $\Rbar$-valued polymorphisms. 
{An analogous result has been obtained in \cite{HMV17}, although in a different setting.}
\begin{theorem}
\label{thm:dichotomy-TVCSP}
Let $D$ be a finite set and $\cat{D}$ a finite set of $\Rbar$-valued relations on $D$.
If $(f,0)$ is an $\Rbar$-valued polymorphism of $\cat{D}$ for some Siggers operation $f$ on $D$,
then $\TVCSP(\cat{D})$ is in P.
Otherwise, it is NP-hard.
\end{theorem}

\begin{remark}
Let $\rho\colon [k]\pto A$ be an $\Rbar$-valued relation on a finite set $A$. The $\Rbar$-valued polymorphisms of $\rho$ give rise to the clone on $A$ consisting of all operations $f$ on $A$ such that $(f,0)$ is an $\Rbar$-valued polymorphism of $\rho$. By \cref{lem:f-zero}, this is the clone of polymorphisms of $\{\rho^\alpha\mid\alpha <\infty\}$. 
In addition, we also have the (in general strictly larger) clone on $A$ which consists of all $f$ such that $(f,\alpha)$ is an $\Rbar$-valued polymorphism of $\rho$ for some $\alpha <\infty$. 
See~\cref{prop:Pol-clone-general}.
\end{remark}

\subsection{TVCSPs and continuous optimisation}
\label{subsec:TVCSP-Set}
Finally, we consider TVCSPs in which the domains $D$ are equal to the set $\mathbb{R}$ of real numbers. 
In this case, $\Sol(I): V \pto \mathbb{R}$ for a TVCSP instance $I=(V,\mathbb{R},\cat{C})$ amounts to a function $\mathbb{R}^V \rightarrow \Rbar$ and  
$\Opt(I)$ is the infimum of $\Sol(I)$.
Hence, it can be regarded as a continuous optimisation problem.
In continuous optimisation, \emph{convexity} of a function plays a key role in the design of efficient minimisation algorithms.
We investigate the relationship between convexity and the TVCSP in what follows.
We note that infinite numeric domains such as $D = \mathbb{R}$, $\mathbb{Q}$, and $\mathbb{Z}$ have been studied in the ordinary CSP (see, e.g., \cite{BoM17}).

A function $\rho: \mathbb{R}^k \rightarrow \Rbar$ is called \emph{quasiconvex} if $\max\{ \rho(\mathbf{x}), \rho(\mathbf{y})\} \ge \rho(\lambda \mathbf{x} + (1-\lambda)\mathbf{y})$ for all $\mathbf{x},\mathbf{y} \in \mathbb{R}^k$ and $\lambda \in [0,1]$.
In other words, a function is quasiconvex if and only if every sublevel set of it is convex.
Quasiconvex functions generalise convex functions, and their minimisation algorithms have been studied (e.g., \cite{Kiw01}).
We can capture quasiconvexity by $\Rbar$-valued polymorphisms.
For each $\lambda \in [0,1]$ we define a binary operation $f_{\lambda}: \mathbb{R}^2 \rightarrow \mathbb{R}$ as $f_\lambda(x,y) = \lambda x + (1-\lambda)y$.
It is then immediate from the definition of quasiconvexity that a function $\rho: \mathbb{R}^k \rightarrow \Rbar$ is quasiconvex if and only if $(f_\lambda,0)$ is an $\Rbar$-valued polymorphism of the corresponding $\Rbar$-valued relation $\rho\colon [k]\pto \mathbb{R}$ for all $\lambda \in [0,1]$.

Next we consider a TVCSP associated with linear $\Rbar$-valued relations, where $\rho: [k]\pto \mathbb{R}$ is \defemph{linear} if there exists $\mathbf{w}^\rho=(w^\rho_1,\dots,w^\rho_k)\in \mathbb{R}^k$ such that for all $\mathbf{d}=(d_1,\dots,d_k) \in \mathbb{R}^k$, we have  $\rho(\mathbf{d}) = \sum_{j=1}^k w^\rho_jd_j$.
Given a TVCSP instance $I=(V,\mathbb{R},\cat{C})$ such that $\rho$ is linear for all $(k,\sigma,\rho)\in\cat{C}$,  $\Sol(I)$ is the supremum of finitely many affine functions (or the constant function with the value $\infty$); hence it is a piecewise-linear convex function $\mathbb{R}^V\to \Rbar$.
Its minimisation can be reduced to linear optimisation, provided that each $\Rbar$-valued constraint $(k,\sigma, \rho)\in\cat{C}$ is given as follows: $\sigma$ takes values in $\mathbb{Q}\cup\{\pm\infty\}$ and is given by a list as in \cref{subsec:TVCSP-FinSet}, and $\rho$ is specified by a list $\mathbf{w}^\rho\in\mathbb{Q}^k$. 
Here, \emph{linear optimisation} (also known as \emph{linear program}) is the problem of minimising a linear function subject to a system of linear inequalities, 
and is one of the central problems in mathematical optimisation (e.g., \cite{Van20}).
We can see that $\Opt(I)$ is equal to the value of the following optimisation problem.
\begin{align*}
\begin{array}{lll}
\text{minimise} & \alpha &\\
\text{subject to} & \alpha\geq \sum_{j=1}^k w^\rho_js(x_j)-\sigma(x_1, \dots, x_k) & ((k, \sigma, \rho) \in \cat{C},\, (x_1, \dots, x_k) \in \dom \sigma), 
\end{array}
\end{align*}
where the variables are $s: V \rightarrow \mathbb{R}$ and $\alpha \in \Rbar$.
This is an instance of linear optimisation. Accordingly, it is solvable in polynomial time by a suitable linear optimisation algorithm (e.g., \cite{Sch98}).

\bibliographystyle{eptcs}
\bibliography{generic}

\end{document}